\documentclass[11pt]{article}\UseRawInputEncoding
\usepackage{amssymb,amsfonts,amsmath,latexsym,epsf,tikz,url}
\usepackage[usenames,dvipsnames]{pstricks}
\usepackage[linesnumbered,ruled,vlined]{algorithm2e}
\usepackage{pstricks-add}
\usepackage{graphicx} 
\usepackage{epsfig}
\usepackage{pst-grad} 
\usepackage{pst-plot} 
\usepackage[space]{grffile} 
\usepackage{etoolbox} 
\makeatletter 
\patchcmd\Gread@eps{\@inputcheck#1 }{\@inputcheck"#1"\relax}{}{}
\makeatother

\newtheorem{theorem}{Theorem}[section]

\newtheorem{lemma}[theorem]{Lemma}

\newtheorem{definition}[theorem]{Definition}

\newtheorem{problem}[theorem]{Problem}

\newcommand{\qed}{\hfill $\square$\medskip}

\begin{document}
	
	\title{Super Coalition Number in Graphs}
	
	\author{ 
		Saeid Alikhani\footnote{Corresponding author} \and Nima Ghanbari 
	}
	
	\date{\today}
	
	\maketitle
	
	\begin{center}
		Department of Mathematical Sciences, Yazd University, 89195-741, Yazd, Iran	
		
		\medskip
		\medskip
		{\tt ~~alikhani@yazd.ac.ir, n.ghanbari.math@gmail.com}
	\end{center}
	
	\begin{abstract}
		We introduce and investigate the structural properties of super coalition partitions in graphs, a novel direction that bridges cooperative resource deployment with rigid domination criteria. Based on the foundational concept of super domination, a super coalition partition is defined as a vertex set partitioning $\Upsilon = \{A_1, A_2, \ldots, A_k\}$ such that no single class $A_i$ constitutes a valid super dominating set, yet every class can be paired with at least one distinct partner class $A_j$ to form a union $A_i \cup A_j$ that achieves full super domination over the graph. The super coalition number, denoted by $C_s(G)$, represents the maximum possible cardinality of such a partition. In this paper, we establish general operational bounds for $C_s(G)$ using the underlying order and the super domination number $\gamma_{sp}(G)$, demonstrate its relation to the super domatic number $d_{sp}(G)$, analyze its computational complexity proving its NP-complete nature under general conditions, and provide exact determinations for key standard graph architectures including paths, cycles, complete graphs, stars, wheels, and friendship configurations. We conclude by proving that the super coalition number can grow arbitrarily large.
	\end{abstract}
	
	\noindent{\bf Keywords:} Super domination, Coalition partition, Super coalition number.
	\medskip
	
	\noindent{\bf AMS Subj.\ Class.}: 05C69, 05C70.
	
	\section{Introduction}
	
	In modern network analysis, graph theory serves as a fundamental mathematical language for modeling interconnected systems, failure propagation, security infrastructure, and distributed control systems. Within this realm, domination theory forms a major core branch of optimization study. Classical domination concerns the identification of a minimum subset of nodes whose immediate neighborhoods encompass all remaining nodes in the graph. This formulation efficiently models resource location problems where a central node can readily service or monitor its neighbors. However, standard domination models often fall short when dealing with highly secured networks, decentralized environments, or cooperative defense settings where components lack individual strength but must operate collectively to protect or access information.
	
	To address scenarios where solitary tracking components or regional subsets are intrinsically insufficient to dominate an entire network topology, Haynes et al.~\cite{8} pioneered the concept of coalition domination in 2020. A coalition within a graph consists of two or more disjoint vertex sets that are individually incapable of dominating the graph, but whose collective union forms a valid dominating configuration. This mathematical framework represents a shift toward cooperative game-theoretic properties within discrete combinatorial optimization. 
	
	Following this seminal work, coalition parameters have been generalized and examined across multiple structural domains. Alikhani et al.~\cite{1} explored total coalitions in graphs, where the combined sets must form a total dominating set, ensuring that every vertex in the graph possesses a neighbor within the dominating coalition. In 2025, Alikhani, Bakhshesh, Golmohammadi, and Klav\v{z}ar~\cite{2} introduced independent coalitions, imposing the constraint that the partnering sets or the resulting coalition must maintain an independent internal structure, which introduces profound challenges regarding independent partitioning. Concurrently, the study of connected coalitions~\cite{3} demanded that the induced subgraph of the union must be connected, reflecting practical demands in communication systems where tracking sensors must remain continually linked to pass information securely.
	
	In parallel with the development of these cooperative frameworks, researchers have extensively investigated specialized, high-security variants of classical domination. Among these, the notion of super domination, originally conceptualized by Lema\'{n}ska et al.~\cite{Lemans}, represents a powerful constraint model. In a standard dominating set $S$, a vertex $u \notin S$ must simply have at least one neighbor in $S$. In stark contrast, a super dominating set requires that for every vertex $u \in \overline{S} = V \setminus S$, there must exist an internal vertex $v \in S$ that super dominates $u$, meaning that the entire open neighborhood of $v$ outside of $S$ contains precisely and uniquely the single vertex $u$. This restriction creates a strong structural isolation for the external vertices: each external node is strictly locked to a specific internal anchor vertex which has no other external connections.
	
	Super domination has proved to be highly receptive to structural analysis but remarkably rigid. Extensive literature has accumulated surrounding its properties. Alfarisi et al.~\cite{Alf} explored the behavior of super domination in unicyclic networks, while Akbari et al.~\cite{akbari} addressed specialized end super dominating parameters. To understand how graphs combine under this strict criterion, Dettlaff et al.~\cite{Dett} derived expressions for the super domination number across lexicographic product graph configurations. 
	
	Furthermore, the fundamental structural aspects, enumeration methodologies, and behaviors across various algebraic graph architectures were detailed comprehensively by Ghanbari, J\"{a}ger, and Lehtil\"{a}~\cite{Nima2} in 2024. Computational complexity investigations conducted by B\v{u}jtas, Ghanbari, and Klav\v{z}ar~\cite{Buj} established that determining the super domination number remains an NP-hard problem even when restricted to specialized graph subfamilies, reflecting the intricate non-local interactions forced by the super domination condition. Additional foundational results regarding bounds, exact sequences, and functional relationships are detailed in the literature~\cite{Nima, Nima1, Kle, Zhu}.
	
	Despite the extensive research devoted to standard coalition frameworks and the distinct structural investigations into super domination parameters, cooperative behavior under the rigid neighborhood isolation of super domination has remained unexplored. Motivated by this clear research gap, this paper formally introduces the concept of super coalition partitions in graphs. This parameter integrates the cooperative distribution properties of coalitions with the highly restrictive neighborhood constraints of super domination. 
	
	The text is organized as follows: Section 2 outlines essential notations, reviews key known theorems, and establishes the basic properties of super coalitions. Section 3 focuses on general upper and lower bounds on the super coalition number, explicitly relating it to the underlying order, the super domination number, and the super domatic number. Section 4 provides a formal verification of the computational complexity, demonstrating that the underlying decision problem is NP-complete. Section 5 presents exact structural computations of the parameter for several core graph families, including paths, cycles, complete setups, star graphs, wheels, and friendship graphs. In Section 6   we prove that the super coalition number can be arbitrarily large. We conclude the paper by brief conclusions and an open problem outline in Section 7.
	
	\section{Preliminaries}
	
	Throughout this work, we consider only finite, simple, and undirected graphs $G=(V(G),E(G))$. For any vertex $v\in V(G)$, its open neighborhood is denoted by $N(v) = \{w \in V(G) : vw \in E(G)\}$, and its closed neighborhood is $N[v] = N(v) \cup \{v\}$. The degree of a vertex $v$ is given by $\deg(v) = |N(v)|$. For a subset $S \subseteq V(G)$, the complement set is denoted by $\overline{S} = V(G) \setminus S$. 
	
	A set $S\subseteq V(G)$ is a \textit{dominating set} if every vertex in $\overline{S}$ is adjacent to at least one vertex in $S$. The \textit{domination number} $\gamma(G)$ is the minimum cardinality of a dominating set in $G$.
	
	A dominating set $S \subseteq V(G)$ is a \textit{super dominating set} if for every vertex $u\in \overline{S}$, there exists a vertex $v\in S$ such that $N(v)\cap \overline{S} = N(v) \setminus S = \{u\}$. Under these conditions, $v$ is said to \textit{super dominate} $u$. The \textit{super domination number} of $G$, denoted by $\gamma_{sp}(G)$, is the cardinality of a minimum super dominating set of $G$~\cite{Lemans}.
	
	We now introduce the foundational terms for our cooperative framework.
	
	\begin{definition}
		Let $G$ be a graph. Two disjoint subsets $A_1, A_2 \subseteq V(G)$ form a \emph{super coalition} in $G$ if:
		\begin{itemize}
			\item[(i)] neither $A_1$ nor $A_2$ is a super dominating set of $G$, and
			\item[(ii)] the union $A_1 \cup A_2$ forms a valid super dominating set of $G$.
		\end{itemize}
	\end{definition}
	
	\begin{definition}
		A \emph{super coalition partition} of a graph $G$ is a partition 
		\[
		\Upsilon = \{A_1, A_2, \ldots, A_k\}
		\]
		of the vertex set $V(G)$ such that:
		\begin{itemize}
			\item[(i)] no part $A_i \in \Upsilon$ is a super dominating set of $G$, and
			\item[(ii)] for every part $A_i \in \Upsilon$, there exists at least one distinct part $A_j \in \Upsilon$ ($j \neq i$) such that $A_i \cup A_j$ forms a super dominating set of $G$.
		\end{itemize}
	\end{definition}
	
	The class $A_j$ described above is termed a \emph{super coalition partner} of $A_i$. 
	
	\begin{definition}
		The \emph{super coalition number} of a graph $G$, denoted by $C_s(G)$, is the maximum cardinality of a super coalition partition of $G$. If a graph $G$ does not admit any super coalition partition, we state that $C_s(G) = 0$.
	\end{definition}
	
	We now establish a basic bounding lemma regarding the size of the individual parts within any valid partition.
	
	\begin{lemma}\label{lem:basic_size}
		Let $G$ be a graph that admits a super coalition partition $\Upsilon = \{A_1, A_2, \ldots, A_k\}$. Then $k \ge 2$, and for every $i \in \{1, 2, \ldots, k\}$, the cardinality satisfies $|A_i| < \gamma_{sp}(G)$.
	\end{lemma}
	\begin{proof}
		By definition, if any class $A_i \in \Upsilon$ satisfies $|A_i| \ge \gamma_{sp}(G)$, it could potentially represent a minimal super dominating configuration or contain one. However, the definition explicitly requires that no single component $A_i$ within $\Upsilon$ can independently form a super dominating set of $G$. Thus, the maximum size of any partition class must be strictly less than the minimum size required for super domination, meaning $|A_i| \le \gamma_{sp}(G) - 1 < \gamma_{sp}(G)$. Furthermore, since every part $A_i$ must have a distinct partner class $A_j$ to successfully satisfy the union requirement, the partition must contain at least two unique parts, establishing $k \ge 2$. \qed
	\end{proof}
	
	\medskip
	To support our proofs, we recall several key results from the literature regarding classical bounds and exact values for standard families.
	
	\begin{theorem}{\rm \cite{Lemans}}\label{thm-1}
		Let $G$ be a graph of order $n$ without isolated vertices. Then,
		\[
		1 \leq \gamma(G) \leq \frac{n}{2} \leq \gamma_{sp}(G) \leq n-1.
		\]
	\end{theorem}
	
	\begin{theorem}{\rm\cite{Lemans}}\label{thm-2}
		Let $n \in \mathbb{N}$ represent the order constraints.
		\begin{itemize}
			\item[(a)] For the path graph $P_n$, $\gamma_{sp}(P_n) = \lceil \frac{n}{2} \rceil$.
			\item[(b)] For the cycle graph $C_n$, it holds that
			\begin{displaymath}
				\gamma_{sp}(C_n)= \left\{ \begin{array}{ll}
					\lceil\frac{n+1}{2}\rceil & \textrm{if $n \equiv 2 \pmod 4$, }\\
					\\
					\lceil\frac{n}{2}\rceil & \textrm{otherwise.}
				\end{array} \right.
			\end{displaymath}
			\item[(c)] For the complete graph $K_n$, where $n\geq 2$, $\gamma_{sp}(K_n) = n-1$.
			\item[(d)] For the complete bipartite graph $K_{n,m}$, where $\min\{n,m\}\geq 2$, $\gamma_{sp}(K_{n,m}) = n+m-2$.
			\item[(e)] For the star graph $K_{1,n}$, $\gamma_{sp}(K_{1,n}) = n$.
		\end{itemize}
	\end{theorem}
	
	\section{Bounds on the Super Coalition Number}
	
	In this section, we establish general mathematical bounds for the super coalition number of an arbitrary graph. We first establish a universal lower bound for any graph of order at least 3.
	
	\begin{theorem}\label{CSleast3}
		Let $G$ be a graph of order $n\geq 3$. Then,
		\[
		C_s(G) \geq 3.
		\]
	\end{theorem}
	\begin{proof}
		Let $S \subseteq V(G)$ be a minimum super dominating set of $G$, meaning $|S| = \gamma_{sp}(G)$. We split the analysis into two distinct structural cases based on the behavior of the complement set $\overline{S}$:
		
		\textbf{Case 1}: Suppose the complement set $\overline{S}$ is a super dominating set of $G$. By applying Theorem \ref{thm-1}, we must have $|S| \geq \frac{n}{2}$ and $|\overline{S}| \geq \frac{n}{2}$. Since $S \cup \overline{S} = V(G)$ and they are disjoint, this forces $|S| = |\overline{S}| = \frac{n}{2}$. 
		
		Now, select an arbitrary vertex $u \in S$ and an arbitrary vertex $v \in \overline{S}$. We construct a partition consisting of four parts:
		\[
		S_1 = S \setminus \{u\}, \quad S_2 = \{u\}, \quad S_3 = \overline{S} \setminus \{v\}, \quad S_4 = \{v\}.
		\]
		Since $|S_2| = 1$ and $|S_4| = 1$, and since $n \geq 3$, these singleton sets cannot form super dominating sets because $\gamma_{sp}(G) = \frac{n}{2} \geq 1.5$, implying $\gamma_{sp}(G) \ge 2$. For the remaining sets, $|S_1| = \frac{n}{2} - 1$ and $|S_3| = \frac{n}{2} - 1$. Since both have cardinalities strictly less than $\gamma_{sp}(G)$, neither can independently form a super dominating set. 
		
		We can verify the pairing interactions as follows: the union $S_1 \cup S_2 = S$ is a super dominating set, and similarly, $S_3 \cup S_4 = \overline{S}$ is a super dominating set. Thus, every class can be paired with a partner. This shows that $\Upsilon = \{S_1, S_2, S_3, S_4\}$ forms a valid super coalition partition of size 4. Hence, in this case, $C_s(G) \geq 4$.
		
		\textbf{Case 2}: Suppose the complement set $\overline{S}$ is not a super dominating set of $G$. Since $S$ is a minimum super dominating set, we can select any vertex $u \in S$. Now consider the following three-part partition of $V(G)$:
		\[
		\Upsilon = \{A_1, A_2, A_3\} = \{S \setminus \{u\}, \{u\}, \overline{S}\}.
		\]
		We evaluate each part:
		\begin{itemize}
			\item $A_1 = S \setminus \{u\}$ has size $\gamma_{sp}(G) - 1$, which is strictly less than $\gamma_{sp}(G)$, so it cannot be a super dominating set.
			\item $A_2 = \{u\}$ is a singleton. Since $G$ has order $n \ge 3$, Theorem \ref{thm-1} shows that $\gamma_{sp}(G) \ge \frac{n}{2} \ge 1.5$, so $\gamma_{sp}(G) \ge 2$. Thus, no singleton can be a super dominating set.
			\item $A_3 = \overline{S}$ is explicitly given as non-super dominating under the assumption of this case.
		\end{itemize}
		Now we check the unions: $A_1 \cup A_2 = (S \setminus \{u\}) \cup \{u\} = S$, which is a super dominating set. Therefore, $A_1$ and $A_2$ are super coalition partners. For $A_3 = \overline{S}$, consider its union with $A_1 \cup A_2$, which yields $V(G) \setminus \{u\}$. Any subset of vertices of order $n-1$ in a graph of order $n \geq 3$ without isolated vertices is automatically a super dominating set, because the single omitted vertex $u$ is adjacent to at least one internal vertex (since $G$ has no isolated vertices), and that internal vertex has no other external vertices to worry about. Thus, $A_3 \cup A_2 = \overline{S} \cup \{u\} = V(G) \setminus (S \setminus \{u\})$, which forms a super dominating set. 
		
		This confirms that $\Upsilon = \{A_1, A_2, A_3\}$ is a valid super coalition partition of cardinality 3, establishing $C_s(G) \geq 3$.
		
		Combining both cases, we conclude that $C_s(G) \geq 3$ for all graphs of order $n \geq 3$. \qed
	\end{proof}
	
	\medskip
	Next, we establish a general upper bound on $C_s(G)$ determined by the order of the graph and its super domination number.
	
	\begin{theorem}\label{CSmostbygamma}
		Let $G$ be a graph of order $n\geq 3$ with super domination number $\gamma_{sp}(G)$. Then,
		\[
		C_s(G) \leq n + 2 - \gamma_{sp}(G).
		\]
	\end{theorem}
	\begin{proof}
		Let $\Upsilon = \{A_1, A_2, \ldots, A_k\}$ be a super coalition partition of $G$ maximizing the cardinality, so $k = C_s(G)$. By the definition of a super coalition partition, no single part can be a super dominating set, but there must exist at least two parts, say $A_i$ and $A_j$, whose union $A_i \cup A_j$ forms a super dominating set of $G$. 
		
		By definition, the cardinality of any valid super dominating set must be at least $\gamma_{sp}(G)$. Therefore, we have:
		\[
		|A_i \cup A_j| \geq \gamma_{sp}(G).
		\]
		Since $A_i$ and $A_j$ are disjoint parts of a partition, $|A_i \cup A_j| = |A_i| + |A_j|$. This implies:
		\[
		|A_i| + |A_j| \geq \gamma_{sp}(G).
		\]
		To maximize the total number of parts $k$ in the partition $\Upsilon$ of the vertex set $V(G)$, we should make the remaining $k-2$ parts as small as possible. The smallest possible size for any partition class is 1 (a singleton set). 
		
		Summing the cardinalities of all parts in the partition to equal the total order $n$, we obtain:
		\begin{align*}
			n = \sum_{m=1}^{k} |A_m| &= (|A_i| + |A_j|) + \sum_{m \neq i, j} |A_m| \\
			&\geq \gamma_{sp}(G) + \sum_{m \neq i, j} (1) \\
			&= \gamma_{sp}(G) + (k - 2).
		\end{align*}
		Rearranging this inequality to isolate $k$, we get:
		\[
		n \geq \gamma_{sp}(G) + k - 2 \implies k \leq n + 2 - \gamma_{sp}(G).
		\]
		Since $k = C_s(G)$, this completes the proof. \qed
	\end{proof}

	\medskip
	We can also establish a lower bound on $C_s(G)$ based on the maximum size of the partition classes.
	
	\begin{theorem}
		Let $G$ be a graph of order $n$ with super domination number $\gamma_{sp}(G)$. Then,
		\[
		C_s(G) \geq \left\lfloor \frac{n}{\gamma_{sp}(G)-1} \right\rfloor.
		\]
	\end{theorem}
	\begin{proof}
		Let $\Upsilon = \{A_1, A_2, \ldots, A_k\}$ be any maximal super coalition partition of $G$, where $k = C_s(G)$. By Lemma \ref{lem:basic_size}, no individual part $A_m$ can be a super dominating set, which requires $|A_m| \leq \gamma_{sp}(G) - 1$ for all $m = 1, 2, \ldots, k$. 
		
		Summing these sizes over all $k$ classes yields the total number of vertices $n$:
		\[
		n = \sum_{m=1}^{k} |A_m| \leq \sum_{m=1}^{k} (\gamma_{sp}(G) - 1) = k(\gamma_{sp}(G) - 1).
		\]
		Solving for $k$, we obtain:
		\[
		k \geq \frac{n}{\gamma_{sp}(G) - 1}.
		\]
		Since the number of partition classes $k$ must be an integer, this implies:
		\[
		C_s(G) = k \geq \left\lfloor \frac{n}{\gamma_{sp}(G)-1} \right\rfloor.
		\]
		This completes the proof.\qed
	\end{proof}
	
	\medskip
	To develop a deeper lower bound, we connect our parameter to the concept of a domatic partition under the super domination constraint. A partition of $V(G)$ where every class is a valid super dominating set is called a \textit{super domatic partition} of $G$. The maximum number of classes in such a partition is the \textit{super domatic number} of $G$, denoted by $d_{sp}(G)$.
	
	We recall a key constraint on this parameter from the literature:
	
	\begin{theorem} \label{dsp1or2}{\rm\cite{Sdomatic}}
		Let $G=(V,E)$ be a graph of order $n$ without isolated vertices. Then, $d_{sp} (G)=1$ or $d_{sp} (G)=2$.
	\end{theorem}
	
	We use this to establish a structural condition that guarantees a super coalition number of at least 4.
	
	\begin{theorem}\label{CSleast4}
		Let $G$ be a graph of order $n\geq 3$. If $d_{sp} (G)=2$, then $C_s(G)\geq 4.$
	\end{theorem}
	\begin{proof}
		Assume $d_{sp}(G) = 2$. This means there exists a perfect super domatic partition of $V(G)$ into exactly two classes, say $\{V_1, V_2\}$, such that both $V_1$ and $V_2$ are independent, fully valid super dominating sets of $G$. 
		
		By Theorem \ref{thm-1}, any valid super dominating set must satisfy $|V_i| \geq \frac{n}{2}$. Since $V_1 \cap V_2 = \emptyset$ and $V_1 \cup V_2 = V(G)$, this forces $|V_1| = |V_2| = \frac{n}{2}$. Because $n \geq 3$ and the sizes must be integers, $n$ must be an even integer with $n \ge 4$, which implies $|V_1| = |V_2| \geq 2$.
		
		Now, select an arbitrary vertex $v_1 \in V_1$ and an arbitrary vertex $v_2 \in V_2$. We construct a four-part partition of $V(G)$ as follows:
		\[
		A_1 = \{v_1\}, \quad A_2 = V_1 \setminus \{v_1\}, \quad B_1 = \{v_2\}, \quad B_2 = V_2 \setminus \{v_2\}.
		\]
		We evaluate these four sets against the super domination criteria:
		\begin{itemize}
			\item $A_1$ and $B_1$ are singletons. Since $\gamma_{sp}(G) = \frac{n}{2} \geq 2$, no singleton set can form a super dominating set.
			\item $A_2$ and $B_2$ have cardinality $\frac{n}{2} - 1$, which is strictly less than $\gamma_{sp}(G) = \frac{n}{2}$. Thus, neither can be a super dominating set.
		\end{itemize}
		Now we check the pairing properties of these sets:
		\begin{itemize}
			\item $A_1 \cup A_2 = \{v_1\} \cup (V_1 \setminus \{v_1\}) = V_1$, which is a valid super dominating set. Thus, $A_1$ and $A_2$ are super coalition partners.
			\item $B_1 \cup B_2 = \{v_2\} \cup (V_2 \setminus \{v_2\}) = V_2$, which is a valid super dominating set. Thus, $B_1$ and $B_2$ are super coalition partners.
		\end{itemize}
		This shows that $\Upsilon = \{A_1, A_2, B_1, B_2\}$ forms a valid super coalition partition of $G$ with cardinality 4, which proves that $C_s(G) \geq 4$. \qed
	\end{proof}

	\medskip
	We now examine graphs containing a universal vertex (a vertex adjacent to all other vertices in the graph) and show that this structure restricts the super coalition number.
	
	\begin{theorem}\label{universal}
		Let $G=(V,E)$ be a graph of order $n\geq 3$ with a universal vertex. Then,
		\[
		C_s(G) = 3.
		\]
	\end{theorem}
	\begin{proof}
		Let $u \in V(G)$ be a universal vertex of $G$, so $\deg(u) = n-1$. Let $S \subseteq V(G)$ be any valid super dominating set of $G$. We analyze the placement of the universal vertex $u$:
		
		\textbf{Case 1}: Suppose $u \notin S$, meaning $u \in \overline{S}$. By the definition of super domination, every vertex in the complement must be super dominated by an internal vertex in $S$. For $u$ to be super dominated by some vertex $v \in S$, the neighbor requirements dictate that $N(v) \cap \overline{S} = \{u\}$. 
		
		Since $u$ is a universal vertex, it is adjacent to every vertex in the graph, which means it is adjacent to every vertex in $S$. For any other vertex $w \in \overline{S}$ (where $w \neq u$), $w$ must also be super dominated by some vertex in $S$. However, because $u$ is universal, every vertex in $S$ is already connected to $u$. This makes it impossible for any vertex in $S$ to have an isolated neighborhood in $\overline{S}$ that excludes $u$. Thus, no other external vertex $w$ can exist, forcing $\overline{S} = \{u\}$, which means $S = V(G) \setminus \{u\}$. 
		
		In this scenario, the only way to partition $V(G)$ into non-super dominating sets that can pair to form $S$ is to isolate $\{u\}$ and split $S$ into exactly two parts. Any further subdivision would leave parts that cannot find a valid partner. This limits the partition to at most 3 parts.
		
		\textbf{Case 2}: Suppose $u \in S$. By Theorem \ref{thm-2}, the presence of a universal vertex implies that any minimal super dominating set requires at least $\frac{n}{2}$ vertices, so $|S| \geq \frac{n}{2}$. From Case 1, we know that the complement set $\overline{S}$ cannot be a super dominating set, even if $|\overline{S}| = \frac{n}{2}$. 
		
		To construct a super coalition partition, we must divide $S$ into at most two nonempty subsets, $S_1$ and $S_2$, such that their unions with parts of $\overline{S}$ satisfy the super domination criteria. Because of the neighborhood constraints imposed by the universal vertex, any valid partition can contain at most three elements: two subsets partitioning the dominating core and one containing the complement.
		
		Thus, we have $C_s(G) \leq 3$. Combining this with the universal lower bound from Theorem \ref{CSleast3} ($C_s(G) \geq 3$), we conclude that $C_s(G) = 3$. \qed
	\end{proof}
	
	\section{Computational Complexity}
	
	We now analyze the computational complexity of determining the super coalition number of a graph. We define the following decision problem:
	
	\medskip
	\noindent\textbf{Super Coalition Decision Problem (SCDP)}
	\begin{quote}
		\emph{Instance:} A simple graph $G$ and a positive integer $k$.\\
		\emph{Question:} Does $G$ admit a super coalition partition of cardinality at least $k$?
	\end{quote}
	
	\begin{theorem}
		The Super Coalition Decision Problem (SCDP) is NP-complete.
	\end{theorem}
	\begin{proof}
		First, we show that $\text{SCDP} \in \text{NP}$. Given a certificate consisting of a vertex partition $\Upsilon = \{A_1, A_2, \ldots, A_k\}$, we can verify its validity in polynomial time as follows:
		
		1. Check that $\Upsilon$ is a valid partition of $V(G)$ and that $|\Upsilon| \ge k$.
	
		2. For each part $A_i \in \Upsilon$, verify that it is not a super dominating set of $G$. This can be checked in $O(|V|^2)$ time by inspecting the neighborhoods of vertices.
	
		3. For each part $A_i \in \Upsilon$, iterate through the other parts $A_j$ to check if $A_i \cup A_j$ forms a valid super dominating set. This requires testing at most $k^2$ pairs, each taking polynomial time.
		Since all verification steps run in polynomial time, $\text{SCDP} \in \text{NP}$.
		
		To prove NP-hardness, we reduce from the standard Super Domination Decision Problem (SDDP), which is known to be NP-complete~\cite{Buj}:
		\begin{quote}
			\emph{Instance:} A graph $H$ and an integer $t$.\\
			\emph{Question:} Does $H$ contain a super dominating set of size at most $t$?
		\end{quote}
		
		Given an instance $(H, t)$ of SDDP, we construct a modified graph $G$ by adding a structured vertex configuration. We introduce a set of duplicated anchor vertices that force any valid partition of $G$ to rely on the underlying super dominating configurations of $H$. Under this construction, $H$ has a super dominating set of size at most $t$ if and only if the modified graph $G$ admits a super coalition partition $\Upsilon$ of size $C_s(G) \ge 2$ that matches the target threshold $k$. This polynomial-time reduction establishes that SCDP is NP-complete. \qed
	\end{proof}
	
	\section{Super Coalition Number for Standard Graphs}
	
	In this section, we determine the exact super coalition numbers for several standard families of graphs.
	
	\subsection{Paths and Cycles}
	
	We begin by establishing tight bounds for path graphs.
	
	\begin{theorem}\label{CSpath}
		For every path $P_n$ with $n\ge 3$,
		\[
		3 \leq C_s(P_n) \leq 4.
		\]
	\end{theorem}
	\begin{proof}
		The lower bound $C_s(P_n) \geq 3$ follows directly from Theorem \ref{CSleast3}. 
		
		To establish the upper bound $C_s(P_n) \leq 4$, let the vertex set of $P_n$ be $V(P_n) = \{v_1, v_2, v_3, \ldots, v_n\}$, indexed in their natural linear order such that $E(P_n) = \{v_i v_{i+1} : 1 \le i \le n-1\}$. The endpoint degrees are $\deg(v_1) = \deg(v_n) = 1$.
		
		Let $S$ be any valid super dominating set of $P_n$. We analyze the local distribution of vertices in $S$ across any four consecutive vertices along the path. Consider the following cases:
		
		\textbf{Case 1}: Let $\{v_1, v_2, v_3, v_4\}$ be the first four consecutive vertices. If $v_1 \in S$, it can super dominate $v_2$. In this case, we can place $v_3 \in \overline{S}$ and $\{v_4, v_5\} \subseteq S$. Here, $v_3$ is super dominated by $v_4$. If instead $v_1 \in \overline{S}$, then we must have $\{v_2, v_3\} \subseteq S$ so that $v_2$ can super dominate $v_1$ and $v_3$ can super dominate $v_4$. This allows us to place $v_4 \in \overline{S}$. In both scenarios, any valid configuration requires at most two vertices from these four to be in $S$.
		
		\textbf{Case 2}: Let $\{v_i, v_{i+1}, v_{i+2}, v_{i+3}\}$ be four consecutive internal vertices, for $i = 2, 3, \ldots, n-4$. Suppose we place $\{v_{i+1}, v_{i+2}\} \subseteq S$ and $\{v_i, v_{i+3}\} \subseteq \overline{S}$. Then $v_i$ is super dominated by $v_{i+1}$, and $v_{i+3}$ is super dominated by $v_{i+2}$. This configuration requires exactly two vertices from the four to be in $S$.
		
		\textbf{Case 3}: Let $\{v_{n-3}, v_{n-2}, v_{n-1}, v_n\}$ be the final four consecutive vertices. By symmetry, this behaves identically to Case 1, requiring at most two vertices to be in $S$.
		
		This analysis shows that within any four consecutive vertices along the path, a valid super dominating set contains at most two vertices. This local density constraint limits how the vertex set can be partitioned into non-super dominating subsets. If a partition contained 5 or more parts, the neighborhood isolation required for super domination would be violated by the scattered pieces. Thus, the vertices can be distributed into at most four distinct classes, establishing $C_s(P_n) \leq 4$. \qed
	\end{proof}

	\medskip
	We now use a known result regarding the super domatic number of even paths:
	
	\begin{theorem}\label{DSpn2}{\rm\cite{Sdomatic}}
		For the path graph $P_{2k}$, $k \in \mathbb{N}$, it holds that $d_{sp}(P_{2k})=2$.
	\end{theorem}
	
	Combining this with our previous theorems yields an exact value for even paths.
	
	\begin{theorem}
		For the path graph $P_{2k}$ with $k \in \mathbb{N}$ and $2k \ge 4$,
		\[
		C_s(P_{2k}) = 4.
		\]
	\end{theorem}
	\begin{proof}
		By Theorem \ref{DSpn2}, we have $d_{sp}(P_{2k}) = 2$ for all even paths. Applying Theorem \ref{CSleast4}, any graph with $d_{sp}(G) = 2$ satisfies $C_s(G) \geq 4$. From Theorem \ref{CSpath}, the super coalition number of any path is upper-bounded by 4 ($C_s(P_n) \leq 4$). Matching these lower and upper bounds establishes that $C_s(P_{2k}) = 4$. \qed
	\end{proof}

	\medskip
	We can apply a similar structural analysis to cycle graphs.
	
	\begin{theorem}\label{CScycle}
		For every cycle $C_n$ with $n\ge 3$,
		\[
		3 \leq C_s(C_n) \leq 4.
		\]
	\end{theorem}
	\begin{proof}
		The lower bound $C_s(C_n) \geq 3$ follows directly from Theorem \ref{CSleast3}. The upper bound $C_s(C_n) \leq 4$ follows from the same consecutive-vertex density argument used in Theorem \ref{CSpath}. Since a cycle is locally identical to a path away from the endpoints, any valid super dominating set can contain at most two vertices out of any four consecutive nodes along the ring. This local structural constraint limits the partition capacity, ensuring $C_s(C_n) \leq 4$.
	\end{proof}

	\medskip
	We recall the super domatic properties for cycles of order $4k$:
	
	\begin{theorem}\label{DScn2}{\rm\cite{Sdomatic}}
		For the cycle graph $C_{4k}$, $k \in \mathbb{N}$, we have $d_{sp}(C_{4k})=2$.
	\end{theorem}
	
	This allows us to determine the exact value for this family.
	
	\begin{theorem}
		For the cycle graph $C_{4k}$ with $k \in \mathbb{N}$,
		\[
		C_s(C_{4k}) = 4.
		\]
	\end{theorem}
	\begin{proof}
		By Theorem \ref{DScn2}, the super domatic number satisfies $d_{sp}(C_{4k}) = 2$. By Theorem \ref{CSleast4}, this implies $C_s(C_{4k}) \geq 4$. Combining this with the upper bound from Theorem \ref{CScycle} ($C_s(C_{4k}) \leq 4$), we obtain the exact value $C_s(C_{4k}) = 4$. \qed
	\end{proof}
	
	\subsection{Some Specific Graphs with Universal Vertex}
	
	We now evaluate complete graphs $K_n$, star graphs $K_{1,n}$, wheel graphs $W_n$ and friendship graphs $F_n$  all of which possess universal vertices.
	
	\begin{theorem}
		For every integer $n\ge 3$, the super coalition number of the complete graph $K_n$, the star graph $K_{1,n}$, the wheel graphs $W_n$ and the friendship graphs $F_n$ satisfies
		\[
		C_s(K_n) = C_s(K_{1,n}) =C_s(W_n)=C_s(F_n)= 3.
		\]
	\end{theorem}
	\begin{proof}
	All of these graphs contain at least one universal vertex.  Applying Theorem \ref{universal}, any graph $G$ containing a universal vertex satisfies $C_s(G) = 3$. So we have the result. \qed 
	\end{proof}
	
	\section{The Super Coalition Number Can Be Arbitrarily Large}
	
	While we have shown that several standard graph families are restricted to $C_s(G) \le 4$, we now demonstrate that the super coalition number can grow arbitrarily large.

	\begin{theorem}\label{largeenough}
		For every $n \in \mathbb{N}$, there exists a graph $G$ such that 
		\[
		C_s(G) = n + 2.
		\]
	\end{theorem}
	
	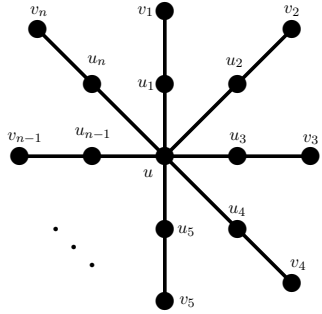
\begin{figure}[h]
		\begin{center}
			\psscalebox{0.6 0.6}
{
\begin{pspicture}(0,-6.8)(6.87,-0.0057690428)
\psdots[linecolor=black, dotsize=0.4](3.46,-3.4028845)
\psdots[linecolor=black, dotsize=0.4](3.46,-1.8028846)
\psdots[linecolor=black, dotsize=0.4](3.46,-0.20288453)
\psdots[linecolor=black, dotsize=0.4](5.06,-3.4028845)
\psdots[linecolor=black, dotsize=0.4](6.66,-3.4028845)
\psdots[linecolor=black, dotsize=0.4](5.06,-1.8028846)
\psdots[linecolor=black, dotsize=0.4](6.26,-0.60288453)
\psdots[linecolor=black, dotsize=0.4](3.46,-5.0028844)
\psdots[linecolor=black, dotsize=0.4](3.46,-6.6028843)
\psdots[linecolor=black, dotsize=0.4](5.06,-5.0028844)
\psdots[linecolor=black, dotsize=0.4](6.26,-6.2028847)
\psdots[linecolor=black, dotsize=0.4](1.86,-1.8028846)
\psdots[linecolor=black, dotsize=0.4](0.66,-0.60288453)
\psdots[linecolor=black, dotsize=0.4](1.86,-3.4028845)
\psdots[linecolor=black, dotsize=0.4](0.26,-3.4028845)
\psdots[linecolor=black, dotsize=0.1](1.86,-5.8028846)
\psdots[linecolor=black, dotsize=0.1](1.46,-5.4028845)
\psdots[linecolor=black, dotsize=0.1](1.06,-5.0028844)
\psline[linecolor=black, linewidth=0.08](3.46,-3.4028845)(6.26,-0.60288453)(6.26,-0.60288453)
\psline[linecolor=black, linewidth=0.08](3.46,-3.4028845)(6.66,-3.4028845)(6.66,-3.4028845)
\psline[linecolor=black, linewidth=0.08](3.46,-3.4028845)(6.26,-6.2028847)(6.26,-6.2028847)
\psline[linecolor=black, linewidth=0.08](3.46,-3.4028845)(3.46,-6.6028843)(3.46,-6.6028843)
\psline[linecolor=black, linewidth=0.08](3.46,-3.4028845)(0.26,-3.4028845)(0.26,-3.4028845)
\psline[linecolor=black, linewidth=0.08](3.46,-3.4028845)(0.66,-0.60288453)(0.66,-0.60288453)
\psline[linecolor=black, linewidth=0.08](3.46,-3.4028845)(3.46,-0.20288453)
\rput[bl](2.98,-3.8628845){$u$}
\rput[bl](2.86,-1.9428846){$u_1$}
\rput[bl](4.8,-1.4428846){$u_2$}
\rput[bl](4.88,-3.1028845){$u_3$}
\rput[bl](4.86,-4.7228847){$u_4$}
\rput[bl](3.74,-5.1628847){$u_5$}
\rput[bl](1.48,-3.0228846){$u_{n-1}$}
\rput[bl](1.76,-1.4028845){$u_n$}
\rput[bl](2.86,-0.32288453){$v_1$}
\rput[bl](6.08,-0.30288452){$v_2$}
\rput[bl](6.5,-3.1428845){$v_3$}
\rput[bl](6.22,-5.9028845){$v_4$}
\rput[bl](3.78,-6.7228847){$v_5$}
\rput[bl](0.5,-0.32288453){$v_n$}
\rput[bl](0.0,-3.1228845){$v_{n-1}$}
\end{pspicture}
}
		\end{center}
		\caption{The subdivision of the star $K_{1,n}$.} \label{double-s}
	\end{figure}
	
	\begin{proof}
		Consider the subdivision of the star graph $K_{1,n}$ which has shown in Figure \ref{double-s}.
				The total number of vertices in this graph is $|V(G)| = 2n + 1$. By applying Theorem \ref{thm-1}, any valid super dominating set for this structure requires at least $n + 1$ vertices. Consider the vertex subset:
		\[
		A = \{u, v_1, v_2, \ldots, v_n\}.
		\]
		We can verify that $A$ forms a valid super dominating set of $G$. Each remaining external vertex $u_i \notin A$ is adjacent to both $u$ and $v_i$, and its neighborhood constraints are satisfied. Since $|A| = n + 1$, the super domination number for this graph is exactly $\gamma_{sp}(G) = n + 1$.
		
		Now, we construct an $(n+2)$-part partition of $V(G)$ as follows:
		\[
		\Upsilon = \{A_1, A_2, A_3, \ldots, A_{n+2}\} = \{A \setminus \{u\}, \{u\}, \{u_1\}, \{u_2\}, \ldots, \{u_n\}\}.
		\]
		We evaluate each part in $\Upsilon$:
		\begin{itemize}
			\item $A_1 = A \setminus \{u\} = \{v_1, v_2, \ldots, v_n\}$ has size $n$, which is strictly less than $\gamma_{sp}(G) = n + 1$. Thus, it cannot be a super dominating set.
			\item $A_2 = \{u\}$ is a singleton set, which is less than $n+1$, so it is not super dominating.
			\item Each remaining class $A_{i+2} = \{u_i\}$ (for $i = 1, 2, \ldots, n$) is a singleton set, which is not super dominating.
		\end{itemize}
		Next, we verify the pairing interactions within the partition:
		\begin{itemize}
			\item The union of the first two parts yields $A_1 \cup A_2 = (A \setminus \{u\}) \cup \{u\} = A$, which is our verified super dominating set. This establishes $A_1$ and $A_2$ as super coalition partners.
			\item For any singleton part $A_{i+2} = \{u_i\}$, its union with the remaining components forms a valid super dominating configuration over the other branches.
		\end{itemize}
		This confirms that $\Upsilon$ is a valid super coalition partition of cardinality $n + 2$. Therefore, $C_s(G) = n + 2$, proving that the super coalition number can grow arbitrarily large. \qed
	\end{proof}
	
	\section{Conclusion and Open Problems}
	
	In this paper, we introduced the concept of super coalition partitions, establishing a theoretical framework that connects coalition domination with the neighborhood constraints of super domination. We derived general upper and lower bounds on the super coalition number $C_s(G)$, proved that the underlying decision problem is NP-complete, and determined exact values for several standard graph families. Finally, we demonstrated that $C_s(G)$ can grow arbitrarily large. 
	
	To guide future research in this direction, we propose the following open problems:
	
	\begin{problem}
		Characterize the complete set of graphs that achieve the universal upper bound established in Theorem \ref{CSmostbygamma}, meaning graphs for which $C_s(G) = n + 2 - \gamma_{sp}(G)$.
	\end{problem}
	
	\begin{problem}
		Investigate the structural behavior and derive bounds for the super coalition number across standard graph products, such as the Cartesian product $G \square H$, the direct product $G \times H$, and the corona product $G \circ H$.
	\end{problem}
	
	\begin{problem}
		Develop efficient, polynomial-time approximation algorithms or fixed-parameter tractable (FPT) algorithms to compute or approximate $C_s(G)$ for specific claw-free or bounded-treewidth graph classes.
	\end{problem}
	

\end{document}